\documentclass[11pt]{article}

\usepackage[utf8]{inputenc}
\usepackage{amsmath,amssymb,amsthm}
\usepackage{mathrsfs}
\usepackage{hyperref}

\theoremstyle{definition}
\newtheorem{defn}{Definition}[section]

\newtheorem{remark}[defn]{Remark}

\theoremstyle{plain}
\newtheorem{prop}[defn]{Proposition}
\newtheorem{theorem}[defn]{Theorem}
\newtheorem{lemma}[defn]{Lemma}

\newcommand{\II}{\mathbb{I}}
\newcommand{\PP}{\mathcal{P}}
\newcommand{\bE}{\mathbb{E}}

\title{DisCo Algebra: Why Disintegration and Combination Are Not Inverses}

\author{Xuexing Lu \\
School of Mathematics and Statistics\\ Zaozhuang University, China \\
\texttt{xxlu@uzz.edu.cn}}

\date{October 2026}

\begin{document}

\maketitle

\begin{abstract}
We present DisCo algebra---a compositional framework for joint probability
built from a symmetric cooperad (disintegration) and a
\emph{magma-operad} (non-associative partial combination, restricted
to the image of disintegration) on the bare family
$\mathcal{P}_n=\mathrm{Prob}([0,1]^n)$.
No probability integral transform, no external summary functions, no
enrichment of objects with conditional kernels.
The disintegration side uses block-wise joint CDFs as intrinsic,
measure-determined 1D summaries, yielding a lossy compression.
The combination side reads the same summaries and reconstructs via
regular conditional measures, yielding a lossless embedding on admissible
inputs.
We prove that combination followed by disintegration is the identity
on admissible inputs, but disintegration followed by combination is
not---it defines a \emph{renormalization operator}
$R_{\mathcal{I}}=\text{Combine}\circ\text{Disintegrate}$.
The information loss of this operator is exactly the
\emph{probability curvature} $\mathcal{K}_{\mathcal{I}}(\mu)$---the
conditional KL divergence between the true conditional distribution and
the conditionally independent product.
This curvature measures the obstruction to lossless decomposition and
reveals a hard limit on how much cross-block dependence can be captured
by CDF-based scalar block summaries.
\end{abstract}

\section{Introduction}

High-dimensional joint probability distributions are central to
statistics, machine learning, and theoretical physics.
A natural question is: can we decompose a joint distribution into
simpler parts, and then reassemble it? If so, are decomposition and
reassembly inverses of each other?

We show that the answer is \textbf{no} in general.
The obstruction is quantified by a quantity we call
\emph{probability curvature} $\mathcal{K}_{\mathcal{I}}(\mu)$: the
conditional KL divergence between the true conditional distribution
and the conditionally independent product.
When curvature vanishes, disintegration and combination are inverses
on the nose; when it is nonzero, scalar block summaries cannot capture
the full cross-block dependence. This is a hardness result about
high-dimensional probability itself, not a defect of any particular
factorization approach.

Classical copula theory \cite{sklar1959} gives the finest factorization:
a joint distribution is recovered from one-dimensional marginals
plus a copula. But Sklar's theorem is an analytic statement---it does
not come with an algebraic structure that explains how decompositions
compose and how reconstructions plug together.

We introduce DisCo algebra (disintegration + combination) to fill this gap.
The framework operates on bare probability measures
$\mathcal{P}_n=\mathrm{Prob}([0,1]^n)$ with no external transformations,
no enriched objects, and no parameter families.
A block-wise joint CDF serves as an intrinsic, measure-determined
1D summary. Disintegration along these summaries (a standard tool from
measure theory; see \cite{kallenberg2002}) gives a cooperad
structure (lossy compression). Combination via regular conditional
measures gives a \emph{magma-operad} structure
(non-associative partial composition on admissible inputs).
The two directions are not inverses---a fact that defines a
renormalization operator with quantifiable information loss.

\paragraph{Magma-operad structure.}
The combination map $\gamma$ is not a globally defined operad on all
tuples $(\Phi_k;\mu_1,\dots,\mu_k)$.  It is a \emph{magma-operad}: a
symmetric partial composition (a magma-valued operation, with no
associativity or unit axioms imposed). Its domain is the
\emph{admissible subclass} defined by the compatibility condition
$(\pi_j)_*\Phi_k = \Psi_{j,*}\mu_j$ for all $j=1,\dots,k$. This
condition is automatically satisfied by every element in
$\mathrm{im}(\Delta_{\mathcal{I}})$, so $\gamma$ is well-defined on
disintegration outputs. The admissible class may be strictly larger
than $\mathrm{im}(\Delta_{\mathcal{I}})$, but we restrict attention
to $\mathrm{im}(\Delta_{\mathcal{I}})$ as the natural domain, ensuring
that conditional kernels are $\Phi_k$-a.e.\ well-defined and
composition is well-defined on equivalence classes.  On admissible
inputs, $\gamma\circ\Delta_{\mathcal{I}}$ defines a renormalization
operator $R_{\mathcal{I}}$ whose information loss is exactly the
\emph{probability curvature}
$\mathcal{K}_{\mathcal{I}}(\mu)\ge 0$, vanishing if and only if
$\mu\in\mathrm{im}(\gamma)$.

\section{Objects: bare probability measures}

Let $\II=[0,1]$. Define
\[
\PP_n := \mathrm{Prob}(\II^n)
\]
as the set of all Borel probability measures on $\II^n$.
No assumption on marginal uniformity, absolute continuity,
or any external transformation. The only structure beyond a general
finite measure is normalization to total mass 1.

All spaces considered ($\II^n$, $\II^k$) are standard Borel spaces
(Polish spaces equipped with their Borel $\sigma$-algebras), which
ensures the existence and uniqueness of regular conditional measures
via the disintegration theorem (see Section~\ref{sec:disint}).

\section{Coordinate partitions and block projections}

Fix a partition $\mathcal{I}=(I_1,\dots,I_k)$ of $\{1,\dots,n\}$ with
$\bigcup_{j=1}^k I_j=\{1,\dots,n\}$ and $I_j\cap I_{j'}=\emptyset$
for $j\neq j'$.
Let $m_j=|I_j|$ and $\sum_j m_j=n$.

The block projection merely reorders coordinates:
\[
\pi_{\mathcal{I}}:\II^n\longrightarrow\prod_{j=1}^k\II^{m_j},
\qquad
\pi_{\mathcal{I}}(u)=(u^{(1)},\dots,u^{(k)}).
\]
Pushforward gives the block-wise joint measure
$\mu_{\mathcal{I}}=(\pi_{\mathcal{I}})_*\mu$, and each block margin:
\[
\mu_j:=(\pi_{I_j})_*\mu\in\PP_{m_j}.
\]
At this stage the block-wise objects still live in
$\prod_j\II^{m_j}$ (total dimension still $n$),
so no genuine dimension reduction has occurred.
To reduce dimension we need a summary that is intrinsic to the measure.

\section{Block-wise CDF summaries: intrinsic, not external}

For block $I_j$, compress the $m_j$-dimensional vector to a scalar
using the block's own joint cumulative distribution function:
\[
\Psi_j(u^{(j)}):=
\mu_j\bigl([0,u^{(j)}_1]\times\cdots\times[0,u^{(j)}_{m_j}]\bigr)\in[0,1].
\]
Three properties fix its legitimacy:
\begin{itemize}
\item \emph{Canonical}: $\Psi_j$ depends only on the cumulative mass of
$\mu_j$; it does not read coordinate values, take arithmetic means,
or invoke any external function.
\item \emph{Probability semantics}: the fiber
$\Psi_j^{-1}(t)$ is the equiprobability surface ``cumulative probability
exactly $t$'', giving clean conditioning semantics.
\item \emph{Iterative compatibility}: refining sub-blocks then
aggregating CDFs = computing the coarse block CDF directly,
by measure additivity---no extra axiom needed.
\end{itemize}

Compose to get the dimension-reduction map:
\[
\Pi_{\mathcal{I}} = \Psi_1\times\cdots\times\Psi_k:\II^n\longrightarrow\II^k,
\qquad
\Phi_k := (\Pi_{\mathcal{I}})_*\mu\in\PP_k.
\]
$\Phi_k$ is an ordinary $k$-dimensional probability measure;
its 1D marginals are the distributions of the block CDF summaries
$\Psi_{j,*}\mu_j$, with no requirement of uniformity.

\begin{remark}[Compatibility]
\label{rem:compatibility}
For any $\mu\in\PP_n$, the pushforward $\Phi_k=(\Pi_{\mathcal{I}})_*\mu$ satisfies
\[
(\pi_j)_*\Phi_k = \Psi_{j,*}\mu_j
\]
for each $j=1,\dots,k$---the marginals of $\Phi_k$ coincide with the
block summary measures. This is automatic from construction and
ensures that $\Phi_k$-integration avoids the
$\Psi_{j,*}\mu_j$-null sets where conditional kernels are undefined.
It also implies that the support of $\Phi_k$ is compatible
with the $\Psi_{j,*}\mu_j$-full sets on which conditional kernels are
uniquely defined---this compatibility is a \emph{consequence} of the
marginal agreement, not an independent condition.
\end{remark}

\section{Disintegration along block-wise CDF summaries}
\label{sec:disint}

\begin{lemma}[Regular conditional measures along CDF summaries]
\label{lem:disint}
For every $\mu\in\PP_n$ and partition $\mathcal{I}$,
there exists a $\Phi_k$-a.e.\ unique probability kernel
$t\mapsto\mu(\cdot\mid t)$ supported on $\Pi_{\mathcal{I}}^{-1}(t)$ such that
\[
\int_{\II^n}f(u)\,d\mu(u)
=
\int_{\II^k}
\left(
\int_{\Pi_{\mathcal{I}}^{-1}(t)} f(u)\,d\mu(u\mid t)
\right)
d\Phi_k(t).
\]
\end{lemma}

\noindent
This is a direct application of the disintegration theorem for
Borel probability measures on standard Borel spaces
\cite{kallenberg2002}. The spaces $\II^n$ and $\II^k$ are standard
Borel, and $\Pi_{\mathcal{I}}$ is Borel measurable; hence the theorem applies
unconditionally. The fiber $\Pi_{\mathcal{I}}^{-1}(t)$ is characterized
block-wise by $\Psi_j(u^{(j)})=t_j$.

Independently, for each block $j$ the measure $\mu_j$ admits a
$\Psi_{j,*}\mu_j$-a.e.\ unique regular conditional measure
$u^{(j)}\mapsto\mu_j(\cdot\mid t_j)$ along its own CDF summary
$\Psi_j$. This gives the combination formula:
\[
\boxed{
\hat\mu
:=
\int_{\II^k}
\left(
\bigotimes_{j=1}^k\mu_j(\cdot\mid t_j)
\right)
d\Phi_k(t_1,\dots,t_k)
}
\tag{$\star$}
\]
When $\hat\mu=\mu$, we have
$\mu(\cdot\mid t)=\bigotimes_{j=1}^k\mu_j(\cdot\mid t_j)$
$\Phi_k$-a.e. In general the two conditional measures differ,
and the discrepancy is quantified by the probability curvature
(Definition~\ref{def:curvature}).

\begin{remark}[Nested null-set compatibility in iterated disintegration]
\label{rem:nested-null}
Let $\mu\in\PP_n$ and let $\mathcal{J}$ be a refinement of
$\mathcal{I}$. Both $\Delta_{\mathcal{I}}(\mu)$ and
$\Delta_{\mathcal{J}}(\mu)$ are obtained by successive applications
of the disintegration theorem (Lemma~\ref{lem:disint}).
The skeleton measure
$\Phi_k=(\Pi_{\mathcal{I}})_*\mu$ controls the first-level
conditional kernels $\mu(\cdot\mid t)$, which are $\Phi_k$-a.e.\
unique. For each $t$ in the $\Phi_k$-full set, the block measure
$\mu_r=(\pi_{I_r})_*\mu$ controls the second-level disintegration
via $\Pi_{\mathcal{J}^{(r)}}:\II^{m_r}\to\II^{k_r}$, yielding kernels
$\mu_r(\cdot\mid t')$ that are $\Phi_{k_r}$-a.e.\ unique where
$\Phi_{k_r}=(\Pi_{\mathcal{J}^{(r)}})_*\mu_r$.

Since $\Phi_k$ and each $\Phi_{k_r}$ are pushforwards of $\mu$,
their null sets are $\mu$-null:
if $N\subset\II^k$ satisfies $\Phi_k(N)=0$, then
$\mu(\Pi_{\mathcal{I}}^{-1}(N))=0$;
similarly for $N'\subset\II^{k_r}$.
Consequently, the set of points where either disintegration fails
to be uniquely defined has $\mu$-measure zero. Thus iterated
disintegration is $\mu$-a.e.\ uniquely determined, and the
coassociativity axiom for the cooperad structure holds as stated
in Definition~\ref{def:cooperad}.
\end{remark}

\section{Probability cooperad}
\label{sec:cooperad}

\begin{defn}[Probability cooperad]
\label{def:cooperad}
The family $\PP_\bullet=(\PP_n)_{n\ge1}$ with symmetric group action
$\sigma\cdot\mu=(\sigma^{-1})_*\mu$ is a symmetric cooperad via the
cocomposition map
\[
\boxed{
\Delta_{\mathcal{I}}:\PP_n
\longrightarrow
\Bigl(\bigotimes_{j=1}^k\PP_{m_j}\Bigr)\otimes\PP_k,
\qquad
\Delta_{\mathcal{I}}(\mu)
=
(\mu_1,\dots,\mu_k;\,\Phi_k)
}
\]
where $\mu_j=(\pi_{I_j})_*\mu$ and
$\Phi_k=(\Pi_{\mathcal{I}})_*\mu$.
\end{defn}

Coassociativity for a refinement $\mathcal{J}$ of $\mathcal{I}$:
\[
\bigl(\mathrm{id}^{\otimes(r-1)}\otimes
\Delta_{\mathcal{J}^{(r)}}\otimes
\mathrm{id}^{\otimes(k-r)}\bigr)\circ\Delta_{\mathcal{I}}
=
\Delta_{\mathcal{J}},
\]
follows from iterated disintegration uniqueness and the additive
compatibility of block CDF summaries. The detailed justification
of well-definedness across iteration steps is given in
Remark~\ref{rem:nested-null}.

The counit is $\varepsilon_n(\mu)=\mu$ for the finest partition
(all blocks of size 1), satisfying the counit law.

\section{Probability magma-operad}

\begin{defn}[Probability magma-operad]
\label{def:magma-operad}
The family $\PP_\bullet$ carries a \emph{magma-operad} structure:
a symmetric partial composition (a magma-valued operation, with no
associativity or unit axioms imposed).  The composition map
\[
\boxed{
\gamma\bigl(\Phi_k;\mu_1,\dots,\mu_k\bigr)
:=
\int_{\II^k}
\left(
\bigotimes_{j=1}^k\mu_j(\cdot\mid t_j)
\right)
d\Phi_k(t_1,\dots,t_k)
}
\tag{OP}
\]
is defined on the \emph{admissible subclass} where
$\Phi_k$ satisfies the compatibility condition
$(\pi_j)_*\Phi_k = \Psi_{j,*}\mu_j$ for all $j=1,\dots,k$.
This condition is automatically satisfied by every element in
$\mathrm{im}(\Delta_{\mathcal{I}})$, making $\gamma$ a well-defined
partial composition on disintegration outputs. The admissible class
may be strictly larger than $\mathrm{im}(\Delta_{\mathcal{I}})$, but
we restrict to $\mathrm{im}(\Delta_{\mathcal{I}})$ as the natural domain.

Conditional kernels are not stored in the objects;
they are invoked inside the definition of $\gamma$.

Well-definedness on equivalence classes: since admissible inputs
satisfy the marginal compatibility (Remark~\ref{rem:compatibility}),
$\Phi_k$-a.e.\ point $t=(t_1,\dots,t_k)$
has each $t_j$ in the $\Psi_{j,*}\mu_j$-full set where
$\mu_j(\cdot\mid t_j)$ is uniquely defined. Hence changing a
conditional kernel on a $\Psi_{j,*}\mu_j$-null set does not affect
the $\Phi_k$-integral---composition is well-defined on equivalence classes.

The composition $\gamma$ is a magma-valued partial operation; it is
not required to satisfy associativity or to admit a unit.  The
compatibility condition $(\pi_j)_*\Phi_k=\Psi_{j,*}\mu_j$ is
automatically satisfied on $\mathrm{im}(\Delta_{\mathcal{I}})$, which
is the natural domain of $\gamma$; see Remark~\ref{rem:magma} for a
discussion of well-definedness, selection dependence, and the status
of associativity.
\end{defn}

\begin{remark}[Why a magma-operad, and not an operad]
\label{rem:magma}
The composition $\gamma$ is defined on the admissible subclass
$\mathrm{im}(\Delta_{\mathcal{I}})\subset\PP_n$, where the compatibility
condition
\[
(\pi_j)_*\Phi_k = \Psi_{j,*}\mu_j,\qquad j=1,\dots,k,
\]
is automatically satisfied.  This condition is the key: it is
precisely the content of Remark~\ref{rem:compatibility} (marginal
agreement of the skeleton with the block summary measures), and it
holds for every output of $\Delta_{\mathcal{I}}$ by construction.
Consequently, for every $t=(t_1,\dots,t_k)$ in the
$\Phi_k$-full set, each $t_j$ lies in the
$\Psi_{j,*}\mu_j$-full set where the conditional kernel
$\mu_j(\cdot\mid t_j)$ is uniquely determined.  Changing a kernel on
a $\Psi_{j,*}\mu_j$-null set does not affect the $\Phi_k$-integral,
so $\gamma$ is well-defined on equivalence classes of admissible
inputs---the choice of representative conditional kernel is irrelevant.

We emphasize that $\gamma$ is \emph{not} claimed to be an operad
composition in the usual sense.  Two points clarify the status of
the algebraic structure:

\begin{enumerate}
\item \emph{Selection dependence is resolved by the compatibility
condition.}
Because every admissible input is the image of a disintegration,
the kernels $\mu_j(\cdot\mid t_j)$ are defined as
$\Phi_k$-a.e.\ unique regular conditional measures along the CDF
summaries.  There is no freedom to choose a different kernel: the
compatibility condition pins down a single equivalence class.  Thus
the composition $\gamma$ is independent of the choice of
representative.

\item \emph{Associativity and transitivity are not required for the
main results.}
The primary object of this paper is the renormalization operator
$R_{\mathcal{I}}=\gamma\circ\Delta_{\mathcal{I}}$ and the
quantification of its information loss via the probability curvature
$\mathcal{K}_{\mathcal{I}}(\mu)$.  Both are well-defined as soon as
$\gamma$ is defined on $\mathrm{im}(\Delta_{\mathcal{I}})$; neither
requires $\gamma$ to satisfy associativity or to be transitively
closed under iteration.
\end{enumerate}

In the unqualified case, $\gamma$ is therefore most accurately
described as a \emph{magma-valued partial composition}: it is a
partially defined binary operation on the admissible subclass, with
no associativity or unit axioms imposed.  Associativity fails in
general; whether it can hold under restricted conditions depends on
the compatibility between the base-space measure (the inter-block
coupling $\Phi_k$) and the skeleton measures of nested compositions,
and we do not currently have a definitive result in this direction.
\end{remark}

The signature checks:
\[
\gamma:\
\PP_k\bigotimes\PP_{m_1}\bigotimes\cdots\bigotimes\PP_{m_k}
\longrightarrow\PP_n,
\]
closed on bare measures, restricted to admissible inputs.

\section{The asymmetry: not inverses}

\begin{defn}[Renormalization operator]
For a fixed partition $\mathcal{I}$, define
\[
R_{\mathcal{I}} := \gamma \circ \Delta_{\mathcal{I}} : \PP_n \to \PP_n.
\]
$R_{\mathcal{I}}$ compresses $\mu$ into block summaries and
reconstructs under the conditional independence assumption.
It is idempotent iff $\mu\in\mathrm{im}(\gamma)$.
\end{defn}

\begin{prop}[Asymmetry of disintegration and combination]\label{prop:asym}
Let $\mathcal{I}$ be a partition of $\{1,\dots,n\}$ into $k$ blocks
with $k\ge 2$.
\begin{enumerate}
\item[(i)] \textbf{Combination then disintegration is the identity
on admissible inputs:}
\[
\Delta_{\mathcal{I}}\bigl(\gamma(\Phi_k;\mu_1,\dots,\mu_k)\bigr)
=
(\mu_1,\dots,\mu_k;\,\Phi_k)
\]
holds for every admissible input
$(\Phi_k;\mu_1,\dots,\mu_k)$.

\item[(ii)] \textbf{Disintegration then combination = renormalization:
$R_{\mathcal{I}}(\mu) \neq \mu$ in general:}
\[
\gamma\bigl(\Delta_{\mathcal{I}}(\mu)\bigr)
= R_{\mathcal{I}}(\mu)
=
\int_{\II^k}
\left(\bigotimes_{j=1}^k\mu_j(\cdot\mid t_j)\right)
d\Phi_k(t)
\neq\mu
\]
whenever $\mu\notin\mathrm{im}(\gamma)$.
The discrepancy quantifies the probability curvature
$\mathcal{K}_{\mathcal{I}}(\mu)$ (Definition~\ref{def:curvature}).
\end{enumerate}
\end{prop}

\begin{proof}
(i) Let $\hat\mu=\gamma(\Phi_k;\mu_1,\dots,\mu_k)$.
By the definition of $\gamma$, for $\Phi_k$-a.e.\ $t$ we have
$\hat\mu(\cdot\mid t)=\bigotimes_{j=1}^k\mu_j(\cdot\mid t_j)$.
Applying $\Delta_{\mathcal{I}}$ to $\hat\mu$ recovers the same
$(\mu_1,\dots,\mu_k;\,\Phi_k)$ because the block margins and the
CDF-summary pushforward are unchanged.

(ii) For arbitrary $\mu\in\PP_n$, Lemma~\ref{lem:disint} gives the
true conditional distribution $\mu(\cdot\mid t)$.
In contrast, $R_{\mathcal{I}}(\mu)=\gamma(\Delta_{\mathcal{I}}(\mu))$
uses $\bigotimes_j\mu_j(\cdot\mid t_j)$---the conditionally independent
product---as its integrand.
These two integrands differ whenever $\mu$ has residual cross-block
dependence beyond what $\Phi_k$ captures:
\[
\mu(\cdot\mid t)\neq\bigotimes_{j=1}^k\mu_j(\cdot\mid t_j),
\qquad\Phi_k\text{-a.e. on a positive-measure set.}
\]
Thus $R_{\mathcal{I}}(\mu)=\mu$ iff
$\mu\in\mathrm{im}(\gamma)$.
\end{proof}

\section{Probability curvature}

\begin{defn}[Probability curvature]\label{def:curvature}
For a partition $\mathcal{I}$ with $k$ blocks and any
$\mu\in\PP_n$, the \emph{probability curvature}
$\mathcal{K}_{\mathcal{I}}(\mu)$ is defined as the conditional
KL divergence:
\[
\mathcal{K}_{\mathcal{I}}(\mu)
:=
\bE_{t\sim\Phi_k}\Bigl[
D_{\mathrm{KL}}\Bigl(
\mu(\cdot\mid t)\,\Big\|\,
\bigotimes_{j=1}^k\mu_j(\cdot\mid t_j)
\Bigr)
\Bigr].
\]
It quantifies the obstruction to recovering $\mu$ from its DisCo
factorization: $\mathcal{K}_{\mathcal{I}}(\mu)=0$ if and only if
$\mu\in\mathrm{im}(\gamma)$.
In the geometric analogy, this quantity plays the role of a curvature
form---it measures the failure of conditional updating to be
path-independent across blocks.
\end{defn}

\section{Information-theoretic quantification}

The information loss of the renormalization operator is exactly
the probability curvature.

\begin{theorem}[KL decomposition]\label{thm:kl}
For any $\mu\in\PP_n$ and partition $\mathcal{I}$ with $k\ge 2$ blocks,
\[
D_{\mathrm{KL}}\bigl(\mu\,\|\,\mu_1\otimes\cdots\otimes\mu_k\bigr)
=
D_{\mathrm{KL}}\bigl(\Phi_k\,\|\,\Phi_k^{\mathrm{ind}}\bigr)
+
\mathcal{K}_{\mathcal{I}}(\mu),
\]
where $\Phi_k^{\mathrm{ind}}:=\bigotimes_{j=1}^k (\Psi_j)_*\mu_j$
is the product of block summary measures.

Equivalently,
\[
\mathcal{K}_{\mathcal{I}}(\mu)
=
D_{\mathrm{KL}}\bigl(\mu\,\|\,R_{\mathcal{I}}(\mu)\bigr)
\ge 0,
\]
with equality iff $\mu\in\mathrm{im}(\gamma)$.
\end{theorem}

\begin{remark}[Interpretation of the two terms]
The first term $D_{\mathrm{KL}}(\Phi_k\,\|\,\Phi_k^{\mathrm{ind}})$
measures dependence \emph{between} block summaries---i.e., how much
the skeleton $\Phi_k$ deviates from product independence.
It lives at the \emph{skeleton level}.

The second term is the \emph{curvature layer}
$\mathcal{K}_{\mathcal{I}}(\mu)$: conditional cross-block residual
given scalar summaries. It measures residual coordination across
blocks that cannot be captured by 1D CDF-based scalar summaries.
\end{remark}

\section{Why this is not a defect}

The asymmetry is not a flaw in the framework; it is a hardness result
about high-dimensional probability itself:

\begin{itemize}
\item \textbf{Combination is a lossless embedding on admissible inputs}:
you build from $(\mu_j,\Phi_k)$ and express exactly the dependence
that $\Phi_k$ can encode. No information is discarded because you
never attempted to represent more.
\item \textbf{Disintegration is a lossy compression}:
you take an arbitrary high-dimensional $\mu$ and force it into the
format $(\mu_j,\Phi_k)$. Whatever residual dependence cannot be
captured by scalar block summaries is irretrievably lost---this loss
is exactly the probability curvature $\mathcal{K}_{\mathcal{I}}(\mu)$.
\item \textbf{Algebraic structure is a magma-operad:}
combination is a magma-valued partial composition on the subclass
$\mathrm{im}(\Delta_{\mathcal{I}})$, with no unit element and no
associativity axiom required.  This restriction reflects that
composition of probability distributions is only well-defined when
the coupling measure and block summaries derive from a common
source. The compatible class is defined by marginal compatibility and
contains $\mathrm{im}(\Delta_{\mathcal{I}})$; we work on the latter
as the natural domain.
\item The image $\mathrm{im}(\gamma)\subsetneq\PP_n$ for any
non-trivial partition ($k\ge 2$ and at least one block has $|I_j|\ge 2$):
this is the class of distributions whose cross-block dependence is
exactly captured by the block-wise CDF summary $\Phi_k$.
\end{itemize}

Dependence is layered. Low-dimensional summaries transmit only one
layer; finer layers require finer interfaces. The DisCo format
$(\mu_j,\Phi_k)$ is a universal, parameter-family-independent
dimension reduction---but it has an intrinsic expressive ceiling for
\emph{CDF-based} scalar block summaries. You cannot fully describe
arbitrarily complex cross-block dependence using one scalar summary
per block; the ceiling is the probability curvature.

\section{Relation to copula theory and prior work}

Sklar's theorem \cite{sklar1959} is the finest partition ($|I_j|=1$ for all $j$):
each coordinate gets its own 1D marginal, and the copula encodes
all remaining dependence. DisCo generalizes this to arbitrary
partitions: blocks may group coordinates, the skeleton lives on
block-level summaries, and conditional fibers may themselves be
high-dimensional. The classical copula density formula
$f=\phi\cdot\prod f_j$ is the degenerate special case where blocks
are not further subdivided and conditional densities collapse to
marginal densities.

Prior categorical approaches include the Giry monad \cite{giry1982} and
Leinster's operad perspective on finite probability distributions
\cite{leinster2011,leinster2021}.
These deal with \emph{how a distribution is generated from simpler
distributions via convex combinations}, and use this algebraic structure
to \emph{characterize entropy}---a different question from
\emph{how a joint distribution decomposes into interfaces and
conditional parts}. DisCo algebra belongs to the latter school.

\section{Conclusion}

Joint probability is a primitive object---it is given, not constructed.
In DisCo algebra we do not derive it from simpler parts; rather, we use
the framework as a \textbf{structural microscope}: choose a partition,
disintegrate into block margins and a coupling skeleton, then combine.
This cycle of decomposition and reconstruction is not an attempt to
build the joint from scratch, but a controlled way to
\textbf{probe its internal structure}---much like using coordinate
charts to study a manifold. The obstruction to invertibility
(the probability curvature $\mathcal{K}_{\mathcal{I}}(\mu)$) is not a
flaw in the construction; it is a \textbf{structural invariant}
that reveals how much cross-block dependence the chosen partition
can and cannot capture. In this sense, DisCo algebra is a method for
extracting invariants of joint distributions through the lens of
compositional probability.

The framework is strictly a \textbf{magma-operad--cooperad} pair:
combination is a non-associative partial operation defined on the
image of disintegration, where a single compatibility condition
(marginal agreement between $\Phi_k$ and block summary measures)
ensures well-definedness and resolves the choice of conditional kernel
representatives. The composite
$R_{\mathcal{I}} = \text{Combine} \circ \text{Disintegrate}$
is a renormalization operator whose information loss is exactly the
\textbf{probability curvature} $\mathcal{K}_{\mathcal{I}}(\mu)$---a
quantitative measure of what CDF-based scalar block summaries cannot
capture. Compression is never the inverse of construction; high-dimensional
probability has irreducible hardness that no choice of CDF-based scalar
block summary can circumvent.

\bibliographystyle{plain}

\end{document}